\RequirePackage{fix-cm}
\documentclass[smallextended,envcountsect]{svjour3}       
\smartqed  
\usepackage{graphicx}
\usepackage{array,multirow,booktabs,colortbl} 
\usepackage[export]{adjustbox} 
\usepackage{tikz} 
\usepackage{siunitx} 
\usepackage{xcolor}

\definecolor{Gray}{gray}{0.9}

\journalname{Journal of Scientific Computing}
\usepackage{amsfonts, amssymb,bbm }
\usepackage{graphicx}
\usepackage{mathrsfs,mathtools}
\usepackage{epstopdf}
\usepackage{algorithm}
\usepackage{algpseudocode}
\usepackage{amsfonts}
\usepackage{hyperref}
\usepackage{cleveref}
\usepackage{enumitem}
\usepackage{bm}
\usepackage{ntheorem}
\newcommand{\subparagraph}{}
\usepackage{titlesec}
\usepackage{relsize}

\titlespacing\section{0pt}{12pt plus 4pt minus 2pt}{0pt plus 2pt minus 2pt}
\titlespacing\subsection{0pt}{12pt plus 4pt minus 2pt}{0pt plus 2pt minus 2pt}
\titlespacing\subsubsection{0pt}{12pt plus 4pt minus 2pt}{0pt plus 2pt minus 2pt}

\renewcommand{\arraystretch}{1.5}

\setlist[enumerate,1]{label=(\arabic*),ref=(\arabic*)}

\ifpdf
  \DeclareGraphicsExtensions{.eps,.pdf,.png,.jpg}
\else
  \DeclareGraphicsExtensions{.eps}
\fi

\algdef{SE}[PROCEDURE]{ProcedureTemplate}{EndProcedureTemplate}[3]
{\algorithmicprocedure\ \textproc{#1}%
\ifthenelse{\equal{#2}{}}{}{\textproc{$\langle$\textproc{#2}$\rangle$}}%
\ifthenelse{\equal{#3}{}}{}{(#3)}}%
{\algorithmicend\ \algorithmicprocedure}%

\newcommand{\Rplus}{\protect\hspace{-.1em}\protect\raisebox{.35ex}{\smaller{\smaller\textbf{+}}}}
\newcommand{\Cpp}{\mbox{C\Rplus\Rplus}}

\newcommand{\mathd}{\mathrm{d}}
\newcommand{\R}{\mathbbm{R}}

\theorembodyfont{\upshape}

\title{Fast Barycentric-Based Evaluation Over Spectral/{\em hp} Elements}

\author{Edward Laughton
\and Vidhi Zala 
\and Akil Narayan 
\and Robert M. Kirby 
\and David Moxey}

\institute{Edward Laughton \at College of Engineering, Mathematics and Physical Sciences, University of Exeter, Exeter, UK, \email{el326@exeter.ac.uk} 
\and Vidhi Zala \at Scientific Computing and Imaging Institute and School of Computing, University of Utah, Salt Lake City, UT 84112, \email{vidhi.zala@utah.edu} 
\and Akil Narayan \at Scientific Computing and Imaging Institute and Department of Mathematics, University of Utah, Salt Lake City, UT 84112, \email{akil@sci.utah.edu} 
\and Robert M. Kirby \at Scientific Computing and Imaging Institute and School of Computing, University of Utah, Salt Lake City, UT 84112, \email{kirby@cs.utah.edu} 
\and David Moxey \at Department of Engineering, King's College London, London, UK. \email{david.moxey@kcl.ac.uk}}

\usepackage{amsopn}

\newcommand{\ddx}[1]{\frac{\mathd}{\mathd #1}}

\newcommand{\pfpx}[2]{\frac{\partial #1}{\partial #2}}
\newcommand{\pnfpx}[3]{\frac{\partial^{#3} #1}{\partial #2}}

\newcommand{\N}{\mathbbm{N}}

\newcommand{\bs}[1]{\boldsymbol{#1}}
\renewcommand{\hat}[1]{\widehat{#1}}

\begin{document}

\maketitle
\begin{abstract}
  As the use of spectral/$hp$ element methods, and high-order finite element
  methods in general, continues to spread, community efforts to create
  efficient, optimized algorithms associated with fundamental high-order
  operations have grown.  Core tasks such as solution expansion evaluation at
  quadrature points, stiffness and mass matrix generation, and matrix assembly
  have received tremendous attention.  With the expansion of the types of
  problems to which high-order methods are applied, and correspondingly the
  growth in types of numerical tasks accomplished through high-order methods,
  the number and types of these core operations broaden.  This work focuses on
  solution expansion evaluation at arbitrary points within an element. This
  operation is core to many postprocessing applications such as evaluation of
  streamlines and pathlines, as well as to field projection techniques such as
  mortaring. We expand barycentric interpolation techniques developed on an
  interval to 2D (triangles and quadrilaterals) and 3D (tetrahedra, prisms,
  pyramids, and hexahedra) spectral/$hp$ element methods. We provide efficient
  algorithms for their implementations, and demonstrate their effectiveness
  using the spectral/$hp$ element library {\em Nektar++} by running a series of
  baseline evaluations against the `standard' Lagrangian method, where an
  interpolation matrix is generated and matrix-multiplication applied to
  evaluate a point at a given location. We present results from a rigorous
  series of benchmarking tests for a variety of element shapes, polynomial orders
  and dimensions. We show that when the point of interest is to be repeatedly
  evaluated, the barycentric method performs at worst $50\%$ slower, when
  compared to a cached matrix evaluation. However, when the point of interest
  changes repeatedly so that the interpolation matrix must be regenerated in the
  `standard' approach, the barycentric method yields far greater performance,
  with a minimum speedup factor of $7\times$. Furthermore, when derivatives of
  the solution evaluation are also required, the barycentric method in general
  slightly outperforms the cached interpolation matrix method across all
  elements and orders, with an up to $30\%$ speedup. Finally we investigate a
  real-world example of scalar transport using a non-conformal discontinuous
  Galerkin simulation, in which we observe around $6\times$ speedup in computational
  time for the barycentric method compared to the matrix-based approach. We also
  explore the complexity of both interpolation methods and show that the
  barycentric interpolation method requires $\mathcal{O}(k)$ storage compared to
  a best case space complexity of $\mathcal{O}(k^2)$ for the Lagrangian
  interpolation matrix method.
\end{abstract}

%
\begin{keywords}
 {high-order finite elements, spectral/$hp$ elements, point evaluation, barycentric interpolation}
 \end{keywords}

\section*{Declarations}
\subsection*{Funding}
The first and fifth authors acknowledge support from the EPSRC Platform Grant PRISM under grant EP/R029423/1 and the ELEMENT project under grant EP/V001345/1. 
The second and fourth authors acknowledge support from  ARO W911NF-15-1-0222  (Program Manager Dr. Mike Coyle). The third author acknowledges support from NSF DMS-1848508.
\subsection*{Conflicts of interest}
The authors declare no conflicts of interest.
\subsection*{Data and code availability}
All code is available in the \emph{Nektar++} repository at \url{https://gitlab.nektar.info}.
%
\section{Introduction}
\label{sec:introduction}

The application of high-order numerical methods continues to expand, now spanning aeronautics (e.g., \cite{lombard-2016a}) to biomedical engineering (e.g., \cite{chooi2016intimal}), in large part due to two factors:  
the numerical accuracy combined with low dissipation and dispersion errors they can obtain for certain problem classes \cite{vincent2014}, and 
the computational efficiency they can achieve when balancing approximation power against computational density \cite{Moxey2019}.  
Most high-order finite element, also called spectral/$hp$ element, simulation codes mimic their traditional finite element counterparts in terms of software organization such as the development of local operators that are then ``assembled'' (either in the strict continuous Galerkin sense or in the weak discontinuous Galerkin sense) into a global system that is advanced in some way \cite{DevilleHigh02,Karniadakis2005,Hesthaven07}. 
Tremendous effort has been expended to create optimized elemental operations that evaluate, for instance, the solution expansions at the quadrature points which allows for rapid evaluation of the solution at the points of integration needed for computing the stiffness matrix entries, the forcing terms, and other desirable quantities. 

However, in the case of history points (positions in the field at which one wants to track a particular quantity of interest over time) \cite{Sirisup05}, pathlines/streamlines \cite{steffen2008investigation}, isosurface evaluation, \cite{JallepalliLK} or refinement and mortaring \cite{Laughton21}, evaluation of solution expansions at arbitrary point locations is required. From the perspective of point evaluation over the entire domain, these operations require two phases: given a point in the domain, first finding the element in which that point resides, and then a fast evaluation of the solution expansion on an individual element. Optimization strategies such as octrees and R-trees have been implemented to accelerate the first of these two tasks (e.g., \cite{Jallepalli19}); however, a concerted effort has not been placed on generalized elemental operations such as arbitrary point evaluation.

The purpose of this work is to address the need for efficient arbitrary point evaluation in high-order (spectral/$hp$) element methods. We are building upon the barycentric polynomial interpolation work of Berrut and Trefethen \cite{Trefethen04}, which has been shown to have a myriad of applications within the polynomial approximation world \cite{Kelly17}. In this work, we mathematically generalize barycentric polynomial interpolation to arbitrary simplicial elements, and then demonstrate the effectiveness of this approach on the canonical high-order finite elements shapes in 2D (triangles and quadrilaterals) and 3D (tetrahedra, prisms, pyramids, and hexahedra). The operational efficiency range of our work is designed for polynomial degrees often run within the spectral/$hp$ community -- from polynomial degree $2$ to $10$. The algorithms presented herein are also implemented in the publicly available high-order finite element library {\em Nektar++} \cite{Cantwell2015,MOXEY2020107110}.

The paper proceeds as follows: In Sections \ref{sec:background} and \ref{sec:bary-tensor}, we lay out the mathematical details of our work. In Section \ref{sec:background}, we highlight the one-dimensional building blocks of barycentric interpolation and provide the mathematical framework for considering interpolation over finite element shapes generated through successive application of Duffy transformations of an orthotope element, and then in Section \ref{sec:bary-tensor} we provide details on expanding these ideas to both tensor-product constructed 2D and 3D elements as used in {\em Nektar++}. In Section \ref{sec:implementation}, we present both algorithmic and implementation details. We  provide details that facilitate reproducibility of our results as well as complexity and storage analysis of the proposed strategy. In Section \ref{sec:results}, we present various test cases that demonstrate the efficacy and efficiency of our proposed work, as well as a capstone example that highlights a real-world application of our strategy. We summarize our contributions and conclude with possible future work in Section \ref{sec:conclusions}.

%

\section{Univariate barycentric interpolation}
\label{sec:background}

Berrut and Trefethen \cite{Trefethen04} propose barycentric Lagrange interpolation for stable, efficient evaluation of polynomial interpolants. With $\eta \in [-1,1]$ the independent variable, let $\{z_j\}_{j=0}^k \subset [-1,1]$ denote $k+1$ unique nodes. Defining $Q_k$ as the space of polynomials of degree $k$ or less in one variable, then any $p \in Q_k$ can be represented in its barycentric form,
\begin{align}\label{eq:2}
  p(\eta) &= \frac{\sum\limits_{j=0}^k\frac{ w_j p_j }{(\eta - z_j)}}{\sum\limits_{j=0}^k\frac{w_j}{(\eta-z_j)}} \eqqcolon \frac{N(\eta)}{D(\eta)}, & 
  p_j &\coloneqq p(z_j),
\end{align}
where the weights $w_j$ are given by
\begin{equation}\label{eq:1}
    w_j = \frac{1}{\prod\limits_{{i = 0,i\ne j}}^{k} (z_i - z_j)} \hspace{1cm}  \forall j = 0,1,\cdots,k.
\end{equation}
%
%
The weights are independent of the polynomial $p$, and depend only on the nodal configuration. Barycentric form \eqref{eq:2} reveals that given the values $\{p_j\}_{j=0}^k$ of a polynomial, then evaluation of $p$ at an arbitrary location $\eta$ can be accomplished without solving linear systems or evaluating cardinal Lagrange interpolants. 

In this paper, we focus on the algorithmic advantages of using barycentric form, with proper extensions to some standard multivariate non-tensorial domains that are popular in high-order (spectral/$hp$) finite element methods. These algorithmic advantages stem from the univariate algorithmic advantages: The map $\eta \mapsto p(\eta)$ using the formula \eqref{eq:2} requires $\mathcal{O}(k)$ arithmetic operations, whereas the same map using standard linear expansions frequently requires $\mathcal{O}(k^2)$ operations.

Throughout our discussion, we consider the nodes $z_j$, and subsequently, through \eqref{eq:1}, the barycentric weights $w_j$, as given and fixed. To emphasize the $\mathcal{O}(k)$ complexity of the operations, which we consider in the following section, we define
\begin{align}\label{eq:S-def}
  S_r(\bs{v}, \eta) &\coloneqq \sum_{j=0}^k \frac{v_j w_j}{(\eta - z_j)^r}, & \bs{v} &= \left( v_0, \ldots, v_k \right)^T,
\end{align}
which, given $r, \bs{v}$, ostensibly requires only $\mathcal{O}(k)$ complexity to evaluate $\eta \mapsto S_r(\bs{v},\eta)$. Note in particular that $S_r(\bs{v},\eta)$ is linear in $\bs{v}$ and that 
\begin{align}
  \ddx{\eta} S_r(\bs{v},\eta) = -r S_{r+1}(\bs{v},\eta).
\end{align}
We can write \eqref{eq:2} as 
\begin{align}\label{eq:univariate-barycentric}
  p(\eta) &= \frac{S_1(\bs{p}, \eta)}{S_1(\bs{1}, \eta)}, & \bs{p} &= \left(p_0, \ldots, p_k\right)^T, & \bs{1} &\coloneqq \left(1, \ldots, 1\right)^T \in \R^n,
\end{align}
which clearly demonstrates the $\mathcal{O}(k)$ complexity if $\bs{p}$ is furnished. 

\subsection{Alternatives to barycentric form}
Consider a linear expansion representation of $p \in Q_k$, written in terms of either an (arbitrary) basis $\{\phi_j\}_{j=0}^k$ of $Q_k$, or the cardinal Lagrange functions of the interpolation points $\{z_j\}_{j=0}^k$:
\begin{align}\label{eq:p-expansion}
  p(\eta) &= \sum_{j=0}^k c_j \phi_j(\eta),
  & \mathrm{span}\{\phi_0, \ldots \phi_k \} &= Q_k, 
\end{align}
where the $\phi_j$ basis functions can be, e.g., cardinal Lagrange interpolants, monomials, or orthogonal polynomials:
\begin{subequations}\label{eq:phi-examples}
  \begin{align}\label{eq:phi-lagrange}
    \phi_j(\eta) &= \prod_{\substack{i = 0, \ldots, k\\i \neq j}} \frac{\eta - z_i}{z_j - z_i}, \\\label{eq:phi-monomial}
    \phi_j(\eta) &= \eta^{j-1}, \\\label{eq:phi-opoly}
  \phi_j(\eta) &= \psi_j(\eta),
\end{align}
\end{subequations}
with $\psi_j$ any orthogonal polynomial family, such as Legendre or Chebyshev polynomials. Our main focus is the computational complexity of evaluating a polynomial interpolant given the data values $\{p_j\}_{j=0}^k$. The following diagram summarizes the complexity of utilizing each of these procedures to accomplish the evaluation $\eta \mapsto p(\eta)$:
\begin{align}
  \begin{array}{rccccl}
    \textrm{Barycentric \eqref{eq:2}}: &  \{(z_j)\}_{j=0}^k, \eta & \xrightarrow{\mathcal{O}(k^2)}& \{(w_j,z_j,p_j)\}_{j=0}^k, \eta & \xrightarrow{\mathcal{O}(k)} & p(\eta) \\
    \textrm{Monomial \eqref{eq:p-expansion}, \eqref{eq:phi-monomial}}: & \{(z_j,p_j)\}_{j=0}^k, \eta & \xrightarrow{\mathcal{O}(k^3)} & \{(c_j)\}_{j=0}^k, \eta & \xrightarrow{\substack{\textrm{Naive: }\mathcal{O}(k^2)\\\textrm{Horner: }\mathcal{O}(k)}} & p(\eta) \\
    \textrm{Orth. Poly. \eqref{eq:p-expansion}, \eqref{eq:phi-opoly}}: & \{(z_j,p_j)\}_{j=0}^k, \eta & \xrightarrow{\mathcal{O}(k^3)} & \{(c_j)\}_{j=0}^k, \eta & \xrightarrow{\substack{\textrm{Naive: }\mathcal{O}(k^2)\\\textrm{Clenshaw: }\mathcal{O}(k)}} & p(\eta) \\
    \textrm{Lagrange \eqref{eq:p-expansion}, \eqref{eq:phi-lagrange}}: &                             &                                & \{(z_j, p_j)\}_{j=0}^k, \eta & \xrightarrow{\mathcal{O}(k^2)} & p(\eta) \\
  \end{array}
\end{align}
Above, we have alluded to the fact that direct evaluation of monomial or orthogonal polynomial $(k+1)$-term expansions appears to require $\mathcal{O}(k^2)$ complexity, but clever rearrangement of elements in the summation can lower this to $\mathcal{O}(k)$ in both cases, using either Horner's algorithm (monomials) or Clenshaw's algorithm \cite{clenshaw_note_1955} (orthogonal polynomials). Viewed in this way, there are two advantages to utilizing the barycentric form. The initial one-time computation for the barycentric weights is cheaper than for the monomials. Furthermore, with fixed nodal locations $z_j$, the weights $w_j$ need not be recomputed if $p$ is changed; in other words, the weights $w_j$ do not depend on $p$. Hence if $k$ and the nodes $z_j$ are given and fixed, then the weights $w_j$ can be precomputed and (re-)used for \textit{any} $p \in Q_k$.

Our first task in this paper is to generalize the barycentric procedure to first- and second-derivative.

\subsection{Barycentric evaluation of derivatives}\label{ssec:1d-derivative}
A formula for the derivative of a polynomial can be derived from the barycentric form \eqref{eq:2} that inherits its advantages. We first note two auxiliary expressions for the derivative of the numerator and denominator rational functions,
\begin{align}
  N'(\eta) &= -S_2(\bs{p}, \eta), & D'(\eta) &= -S_2(\bs{1}, \eta),
\end{align}
each of which ostensibly also requires only $\mathcal{O}(k)$ complexity to evaluate if the weights $w_j$ are precomputed. Then, by directly differentiating \eqref{eq:2}, we have
\begin{align}\label{eq:barydiff}
  p'(\eta) = \frac{N'(\eta) - p(\eta) D'(\eta)}{D(\eta)} = \frac{S_2\left(p(\eta) - \bs{p}, \eta\right)}{S_1(\bs{1},\eta)},
\end{align}
where the vector $p(\eta) - \bs{p}$ has entries $p(\eta) - p_j$. Therefore, this ``barycentric" form for $p'(\eta)$ requires (i) an evaluation of $p(\eta)$ that can be accomplished in $\mathcal{O}(k)$ complexity using \eqref{eq:2}, and (ii) an additional $\mathcal{O}(k)$ evaluation of the summation above. Thus, $\eta \mapsto p'(\eta)$ can be evaluated with only $\mathcal{O}(k)$ complexity.

A similar computation shows that 
\begin{align}\label{eq:barydiff2}
  p''(\eta) = \frac{2}{S_1(\bs{1},\eta)} \left( p'(\eta) S_2(\bs{1},\eta) - S_3(p(\eta) - \bs{p}, \eta) \right).
\end{align}
Again, since $\eta \mapsto p(\eta)$ and $\eta \mapsto p'(\eta)$ can be evaluated with $\mathcal{O}(k)$ effort through \eqref{eq:2} and \eqref{eq:barydiff}, some extra $\mathcal{O}(k)$ effort to evaluate $S_2$ and $S_3$ above yields an evaluation $\eta \mapsto p''(\eta)$ that can also be accomplished in $\mathcal{O}(k)$ time.

\section{Tensorization of the barycentric form}\label{sec:bary-tensor}

All the evaluations considered above can be generalized to tensorial formulations, which is the main topic of this section. To that end, we introduce some multidimensional notations. Let $\bs{\eta} \coloneqq(\eta_1, \ldots, \eta_d) \in [1,1]^d$ be a $d$-dimensional vector. Given some multi-index $\bs{k} \in \N_0^d$, we introduce the tensorial space of polynomials $Q_{\bs{k}}$ defined by $\bs{k}$:
\begin{align}
  Q_{\bs{k}} \coloneqq \mathrm{span} \left\{ \bs{\eta}^{\bs{j}} \;\big|\; \bs{j} \in \N_0^d \textrm{ and } \bs{j} \leq \bs{k} \right\},
\end{align}
where we have adopted the standard multi-index notation,
\begin{align}
  \bs{\eta}^{\bs{j}} &\coloneqq \prod_{q=1}^d \eta_q^{j_q}, & \bs{j} &= (j_1, \ldots, j_d) \in \N_0^d,
\end{align}
and $\bs{j} \leq \bs{k}$ is true if all the component-wise inequalities are true.

Fixing $\bs{k}$, we consider the case of representing an element $p$ of $Q_{\bs{k}}$ through its values on a discrete tensorial grid of size $\prod_{q=1}^d k_q$. Like the univariate case, linear expansions in cardinal Lagrange, monomial, and/or orthogonal polynomials are common, but we will exercise the barycentric form. Let a tensorial grid on a $[-1,1]^d$ orthotope be given, with $k_q+1$ points in direction $q$:
\begin{align}
  \left\{z_{j,q}\right\}_{j=0}^{k_q} \subset [-1,1], 
\end{align}
for $q = 1, \ldots, d$. The tensorization of these grids results in the multidimensional grid $\left\{ \bs{\eta}_{\bs{j}} \right\}_{\bs{j} \leq \bs{k}}$ defined by 
\begin{align}
  \bs{z}_{\bs{j}} \coloneqq \left( z_{j_1, 1}, \, z_{j_2, 2}, \, \ldots, \, z_{j_d,d} \right) \in [-1,1]^d.
\end{align}
Given this tensorial configuration of nodes, we define univariate barycentric weights associated with each dimension in a  fashion similar to \eqref{eq:1},
\begin{align}
  w_{q,\ell} &\coloneqq \frac{1}{\prod\limits_{{i = 0,i\ne j}}^{k} (z_{i,q} - z_{j,q})}, & 
  0 &\leq j \leq k_q, &
  1 &\leq q \leq d.
\end{align}
Then, given the data
\begin{align}
  \left\{p_{\bs{j}} \right\}_{\bs{j} \leq \bs{k}}, \hskip 15pt p_{\bs{j}} \coloneqq p\left(\bs{z}_{\bs{j}}\right),
\end{align}
for $p \in Q_{\bs{k}}$, then the multidimensional barycentric form of $p$ is 
\begin{align}\label{eq:bary-multidim}
  p(\bs{\eta}) &= \frac{\sum_{\bs{j} \leq \bs{k}} \frac{p_{\bs{j}} w_{\bs{j}}}{\odot \left( \bs{\eta} - \bs{z}_{\bs{j}}\right)}}
                      {\sum_{\bs{j} \leq \bs{k}} \frac{w_{\bs{j}}}{\odot \left( \bs{\eta} - \bs{z}_{\bs{j}}\right)}}, &
  \odot \left(\bs{\eta} - \bs{z}_{\bs{j}} \right) &\coloneqq \prod_{q=1}^d \left(\eta_q - z_{j_q, q} \right), & 
  w_{\bs{j}} &\coloneqq \prod_{q=1}^d w_{q,j_q}.
\end{align}
Given $\bs{y}\in [-1,1]$, an evaluation of $p(\bs{y})$ above requires $\mathcal{O}\left(\prod_{q=1}^d k_q \right)$ operations, corresponding to the complexity of the summations.

\subsection{Dimension-by-dimension approach}\label{ssec:d-by-d}
Instead of the direct approach \eqref{eq:bary-multidim} for evaluating $p$, we utilize a dimension-by-dimension computation that is slightly more computationally expensive but allows us to directly leverage univariate procedures, greatly simplifying the software implementation. First, consider the functions $\widetilde{p}_q$, for $q = d-1, \ldots, 1$, each of which is a function of $\eta_1, \ldots, \eta_q$ formed by freezing $\eta_{s} = y_{s}$ for $s > q$:
\begin{align}
  \widetilde{p}_d &\coloneqq p, & \widetilde{p}_q(\eta_1, \ldots, \eta_q) &\coloneqq p\left(\eta_1, \ldots, \eta_q, y_{q+1}, \ldots, y_d \right) = \widetilde{p}_{q+1}(\eta_1, \ldots, \eta_q, y_{q+1}).
\end{align}
In order to evaluate $p(\bs{y})$, we will proceed by iteratively constructing $\widetilde{p}_q$ from $\widetilde{p}_{q+1}$ for  $q = d-1, \ldots, 1$. Via barycentric form, ``constructing" $\widetilde{p}_q$ amounts to evaluating this function on the $q$-dimensional tensorial grid 
\begin{align}
  \widetilde{\bs{z}}_{q,\bs{j}} = \left( z_{j_1, 1}, \, z_{j_2, 2}, \, \ldots, \, z_{j_q,q} \right) \in [-1,1]^q.
\end{align}
Then, for example, to first generate $\widetilde{p}_{d-1}$, we must evaluate
\begin{align}
  \widetilde{p}_{d-1}\left(\widetilde{\bs{z}}_{d-1,\bs{j}}\right) = p\left(z_{j_1,1}, \ldots, z_{j_{d-1},d-1}, y_{d}\right),
\end{align}
for every $\bs{j} \leq \left(k_1, \ldots, k_{d-1}\right)$. This can be accomplished via univariate procedures in $\mathcal{O}(k_d)$ complexity since, for each fixed $\bs{j}$,
\begin{align}
  y_d \mapsto p\left(z_{j_1,1}, \ldots, z_{j_{d-1},d-1}, y_{d}\right),
\end{align}
is a polynomial of degree $k_d$, and hence obeys the barycentric formula,
\begin{align}
  p\left(z_{j_1,1}, \ldots, z_{j_{d-1},d-1}, y_{d}\right) &= \frac{S_1\left(\widetilde{\bs{p}}_{\bs{j}}, y_d\right)}{S_1\left(\bs{1}, y_d\right)}, & 
  \widetilde{\bs{p}}_{\bs{j}} \coloneqq \left( p_{(j_1, \ldots, j_{d-1}, \ell)} \right)_{\ell=1}^{k_d}.
\end{align}
In this way, we proceed to iteratively construct $\widetilde{p}_q$, amounting to $\prod_{j=1}^{q} k_j$ evaluations, each of computational complexity $\mathcal{O}(k_{q+1})$,
\begin{align}\label{eq:dim-by-dim}
  \left.
  \begin{array}{ccl}
    p & \xrightarrow{\textrm{$\prod_{j=1}^{d-1} k_j$ $\mathcal{O}(k_d)$ evaluations}} & \widetilde{p}_{d-1} \\
    \widetilde{p}_q & \xrightarrow{\textrm{$\prod_{j=1}^{q-1} k_j$ $\mathcal{O}(k_q)$ evaluations}} & \widetilde{p}_{q-1} \hskip 10pt (2 < q < d) \\ 
    \widetilde{p}_1 & \xrightarrow{\textrm{1 $\mathcal{O}(k_1)$ evaluation}} & p(\bs{y})  
  \end{array}
  \right\} 
\end{align}
In summary, computing $\bs{y} \mapsto p(\bs{y})$ can be accomplished with the procedure above, which entails repeated use of the univariate barycentric form \eqref{eq:univariate-barycentric}. 

As mentioned earlier, the cost of the procedure \eqref{eq:dim-by-dim} is slightly more expensive than direct evaluation of the multidimensional barycentric form \eqref{eq:bary-multidim}. In particular, the ($\bs{k}$-asymptotic) cost of the direct evaluation \eqref{eq:bary-multidim} is $\prod_{q=1}^d k_q$. On the other hand, the dimension-by-dimension approach \eqref{eq:dim-by-dim} incurs additional lower-order costs, and has complexity scaling as,
\begin{align}
  \prod_{q=1}^d k_q + \prod_{q=1}^{d-1} k_q + \ldots = \sum_{j=1}^d \prod_{q=1}^j k_q \leq d \prod_{q=1}^d k_q,
\end{align}
where the inequality is a very crude bound.  Thus, while the algorithm described by \eqref{eq:dim-by-dim} is formally more expensive than direct evaluation \eqref{eq:bary-multidim}, the actual additional cost is relatively small.  In particular, for the physically relevant cases of $d = 2, 3$, this minor increase in cost is acceptable for the achieved gain in implementation ease. 


\subsection{Tensorial functions}
We end this section with a brief remark on a direct simplification that can be employed in the special case that one has prior knowledge that the polynomial $p$ is tensorial, i.e., of the form,
\begin{align}
  p(\bs{\eta}) = \prod_{j=1}^d p_j(\eta_j),
\end{align}
for some univariate polynomials $\{p_j\}_{j=1}^d$ satisfying $\deg p_j \leq k_j$. In many high-order FEM simulations, basis functions in $\bs{\eta}$ space are often tensorial polynomials, so that this situation does occur in practice. In this tensorial case, computing the barycentric weights is simpler since we need to only compute the $k_j$ \textit{univariate} weights for dimension $j$ associated with the grid $z_{q,j}$ for $q \in [k_j]$. Thus, we need to only compute $\sum_{j =1}^d k_j$ weights, as opposed to the full set of $\prod_{j=1}^d k_j$ multivariate weights associated with \eqref{eq:bary-multidim}.

Evaluating $\bs{\eta} \mapsto p(\bs{\eta})$ is likewise faster in this case: once the univariate weights are computed, then each univariate barycentric evaluation $\eta_j \mapsto p_j(\eta_j)$ requires $\mathcal{O}(k_j)$ complexity. Therefore, $\bs{\eta} \mapsto p(\bs{\eta})$ requires only $\mathcal{O}(\sum_{j=1}^d k_j)$ complexity, as opposed to the full multivariate $\mathcal{O}(\prod_{j=1}^d k_j)$ complexity.

\subsection{Derivatives}\label{ssec:multid-deriv}
Section \ref{ssec:d-by-d} discusses how we accomplish the evaluation $\bs{\eta} \mapsto p(\bs{\eta})$ algorithmically by iteratively evaluating along each dimension. This procedure is directly extensible to evaluation of (Cartesian) partial derivatives. For example, if $p \in Q_{\bs{k}}$, suppose for some multi-index $\bs{\lambda} \in \N_0^d$ we wish to evaluate the order-$\bs{\lambda}$ derivative,
\begin{align}
  p^{(\bs{\lambda})} &= \frac{\partial^{|\bs{\lambda}|} p}{\partial \bs{\eta}^{\bs{\lambda}}}, & \partial \bs{\eta}^{\bs{\lambda}} &= \partial x_1^{\lambda_1} \partial x_2^{\lambda_2} \cdots \partial x_d^{\lambda_d}.
\end{align}
We are mostly concerned with $|\bs{\lambda}| \leq 2$, i.e., at most ``second" derivatives, but the procedure we describe applies to derivatives of arbitrary order. The dimension-by-dimension approach can be accomplished with essentially the same procedure as articulated in Section \ref{ssec:d-by-d}: Define
\begin{align}
\begin{split}
  \widetilde{p}_d &\coloneqq p, \\
  \widetilde{p}_q(\eta_1, \ldots, \eta_q) &\coloneqq \frac{\partial \widetilde{p}_{q+1}}{\partial \eta_{q+1}^{\lambda_{q+1}}}(\eta_1, \ldots, \eta_q, y_{q+1}),
\end{split}
\end{align}
so that constructing $\widetilde{p}_q$ from $\widetilde{p}_{q+1}$ at a single grid point requires evaluation of the order-$\lambda_{q+1}$ derivative along dimension $q+1$. Through procedures outlined in section \ref{ssec:1d-derivative}, we can accomplish this in $\mathcal{O}(k_{q+1})$ complexity. We must evaluate the derivative at each of the $\prod_{j=1}^{q} k_j$ grid points associated with dimensions $1, \ldots, q$. Therefore, the outline and complexity of this procedure is precisely as given in \eqref{eq:dim-by-dim}, except that $p(\bs{y})$ should be replaced by $p^{(\bs{\lambda})}(\bs{y})$.

\section{Nontensorial multidimensional formulations via Duffy transformations}
\label{sec:duffy}

The goal of this section is to describe how barycentric interpolation in the tensorial case over $d$-dimensional orthotopes of Section \ref{sec:bary-tensor} can be utilized for efficient evaluation of polynomial approximations in certain nontensorial cases. We focus on (potentially) non-tensorial polynomial approximations in a $d$-dimensional variable $\bs{\xi}$. The variable $\bs{\xi}$ will be related to the tensorial variable $\bs{\eta}$ through ``collapsed coordinates" effected by a Duffy transformation, as described in Section \ref{ssec:collapsed}. The transformation can be applied to a variety of common FEM element types, see Table \ref{tab:elements}. A description of how barycentric evaluation procedures can be used to evaluate polynomials on these potentially nontensorial geometries is given in Section \ref{ssec:poly-eval}; specifically, the procedure is given by \eqref{eq:map-to-eta}. That section also gives precise conditions to which polynomials space $p$ must belong so that the evaluation is exact (Theorem \ref{thm:bary-eval}). Section \ref{ssec:specializations} specializes the evaluation exactness conditions to common element types used in the spectral/$hp$ community. Section \ref{ssec:multid-deriv} closes the section by discussing extension of the evaluation routines to derivative evaluations through use of the chain rule.

\subsection{Collapsed coordinates}\label{ssec:collapsed} 
We consider polynomials in $d$ variables $\bs{\xi}$. The particular type of $\bs{\xi}$-polynomials we consider are defined via a mapping from $\bs{\eta}$ space to $\bs{\xi}$ space, where $\bs{\eta} \in [-1,1]^d$ is the tensorial variable considered in Section \ref{sec:bary-tensor}. The essential building block that allows us to specify the $\bs{\eta} \leftrightarrow \bs{\xi}$ relationship is the Duffy transformation, which is a variable transformation in two dimensions. For $\bs{\eta} \in [-1,1]^d$ and some fixed $i, j \in \{1, \ldots, d\}$ with $i \neq j$, define $D_{i,j}$ as the Duffy transformation that ``collapses" dimension $i$ with respect to or along dimension $j$ and is the identity map on all the other dimensions,
\begin{align}\label{eq:duffy}
  \bs{\zeta} = (\zeta_1, \ldots, \zeta_d) &\coloneqq D_{i,j}(\bs{\eta}),  &
  \zeta_\ell &= { \left\{\begin{array}{ll}
    \frac{1}{2} \left( 1 + \eta_\ell\right) \left(1 - \eta_j \right) - 1, & \ell = i, \\
  \eta_\ell, & \ell \neq i \end{array}\right.},
\end{align}
for $\ell = 1, \ldots, d$. 

\begin{figure}[htbp]
  \centering
    \resizebox{0.9\textwidth}{!}{
      \includegraphics[width=\textwidth]{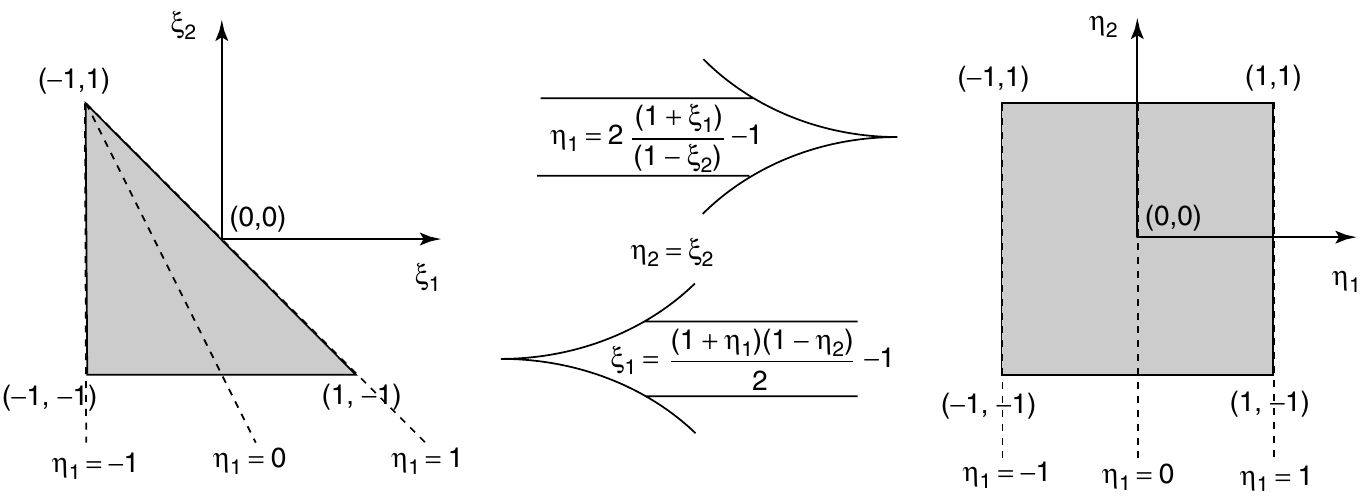}
    }
  
  \caption{Duffy transformations between triangles and quadrilateral reference elements.} 
   \label{fig:duffyColl}
\end{figure}

Various domains in $d$ dimensions can be created by composing Duffy maps. For composing $c < d$ maps, we let $\bs{a} \in [d]^c$ have components that represent the dimensions that are collapsed by a Duffy transformation, and let $\bs{b} \in [d]^c$ have components specifying along which dimensions the collapse occurs. We place the following restrictions on the entries of $\bs{a}$ and $\bs{b}$:
\begin{itemize}
  \item $a_\ell < b_\ell$ for $\ell = 1, \ldots, c$
  \item $a_\ell < b_\ell$ for $\ell = 1, \ldots, c$
  \item $a_\ell < a_{\ell+1}$ for $\ell = 1, \ldots, c-1$.
\end{itemize}
We now define the variable $\bs{\xi}$ as the image of $\bs{\eta}$ under a composition of Duffy transformations defined by $\bs{a}$ and $\bs{b}$,
\begin{align}\label{eq:xi-def}
  \bs{\xi} &\coloneqq D_{\bs{a}, \bs{b}} \coloneqq D_{a_1, b_1} \circ D_{a_{2}, b_{2}} \circ \cdots \circ D_{a_c,b_c} (\bs{\eta}), &
  E(\bs{a}, \bs{b}) &\coloneqq D_{\bs{a}, \bs{b}}\left([-1,1]^d \right).
\end{align}
We are interested primarily in these domains for dimensions $d = 2, 3$. Table \ref{tab:elements} illustrates various standard geometric domains that are the result of particular choices of coordinate collapses.
\begin{table}[h!]
  \begin{center}
  \resizebox{0.95\textwidth}{!}{
    \renewcommand{\tabcolsep}{0.4cm}
    \renewcommand{\arraystretch}{1.3}
    {\scriptsize
      \begin{tabular}{@{}cp{0.3\textwidth}@{}ccc@{}}
      \toprule
                             & Geometric region $E_{\bs{a},\bs{b}}$ & $c$ & $\left\{(a_j, b_j) \right\}_{j=1}^c$ & $g = g_{\bs{a},\bs{b}}$ evaluations\\\midrule 
        \multirow{2}{*}{$d=2$} & Quadrilateral & 0 & --- & $g(1) = \{1\}$, $g(2) = \{2\}$ \\
                             & Triangle & 1 & $(1,2)$  & $g(1) = \{1\}$, $g(2) = \{1, 2\}$ \\
      \rowcolor{Gray}
                             & Hexahedron & 0 & --- & $g(1) = \{1\}$, $g(2) = \{2\}$, $g(3) = \{3\}$ \\
      \rowcolor{Gray}
                             & Prism & 1 & $(1,2)$ & $g(1) = \{1\}$, $g(2) = \{1, 2\}$, $g(3) = \{3\}$ \\
      \rowcolor{Gray}
                             & Tetrahedron & 2 & $(1,2), (2,3)$ & $g(1) = \{1\}$, $g(2) = \{1, 2\}$, $g(3) = \{1, 2, 3\}$ \\
      \rowcolor{Gray}
        \multirow{-4}{*}{$d=3$} & Pyramid & 2 & $(1,3), (2,3)$ & $g(1) = \{1\}$, $g(2) = \{2\}$, $g(3) = \{1, 2, 3 \}$ \\
    \bottomrule
    \end{tabular}
  }
    \renewcommand{\arraystretch}{1}
    \renewcommand{\tabcolsep}{12pt}
  }
  \end{center}
  \caption{Multidimensional domains resulting from collapsed coordinates, multi-indices $\bs{a}$ and $\bs{b}$ identifying the associated multivariate Duffy map, and ancestor functions $g_{\bs{a},\bs{b}}$ defined in \eqref{eq:g-def}.}\label{tab:elements}
\end{table}
A visual example with $d=2$ is also shown in Fig.~\ref{fig:duffyColl} which gives the Duffy transformations between reference triangles and reference quadrilaterals. 

\subsection{Evaluation of polynomials}\label{ssec:poly-eval}
The previous section illustrates how various standard domains that are used to tesselate space in finite element simulations are constructed. This section considers how we can employ barycentric evaluation in $\bs{\eta}$ space to accomplish evaluation of polynomials in $\bs{\xi}$ space. In what follows we assume that the dimension $d$, the grid size multi-index $\bs{k}$, and the Duffy transformation parameters $c$, $\bs{a}$, and $\bs{b}$ are all given and fixed.

Recall that on the tensorial domain $\bs{\eta} \in [-1,1]^d$ we have a tensorial grid $Z_{\bs{k}}$ comprised of $k_q$ points in dimension $q$, resulting in a total of $\prod_{q=1}^d k_q$ points in $Z_{\bs{k}}$. The image of this grid in $\bs{\xi}$ space is the result of applying the Duffy transformation:
\begin{subequations}\label{eq:p-xi-evaluations}
\begin{align}
  \bs{y}_j &\coloneqq D_{\bs{a},\bs{b}}\left(\bs{z}_{\bs{j}} \right), & \bs{j} &\leq \bs{k}.
\end{align}
Assume that data values are furnished from a given function $p$,
\begin{align}
  P_{\bs{k}} &\coloneqq \left( p_{\bs{j}} \right)_{\bs{j} \leq \bs{k}}, & p_{\bs{j}} &\coloneqq p\left(\bs{y}_{\bs{j}}\right), 
\end{align}
\end{subequations}
and are provided on the grid $\bs{y}_{\bs{j}}$. Naturally, these values can be considered as data values in $\bs{\eta}$ space under the (inverse) Duffy transformation,
\begin{align*}
 \left\{ \left( \bs{z}_{\bs{j}}, p_{\bs{j}} \right) \right\}_{\bs{j} \leq \bs{k}},
\end{align*}
and therefore the barycentric routines developed in tensorial form in Section \ref{sec:bary-tensor} can be applied. The main result of this section is the provision of conditions on $p$ in $\bs{\xi}$ space under which applying the tensorial barycentric interpolation procedure in $\bs{\eta}$ space results in an exact evaluation. More precisely, we consider the following algorithmic set of steps given a point $\bs{\xi}$ in $E(\bs{a}, \bs{b})$, and a function $p$:
\begin{align}\label{eq:map-to-eta}
  \left( \bs{\xi}, p \right) \xrightarrow{\eqref{eq:p-xi-evaluations}} \left( \bs{\xi}, Z_{\bs{k}}, P_{\bs{k}} \right) \xrightarrow{\bs{\eta} = D_{\bs{a},\bs{b}}^{-1}(\bs{\xi})} \left( \bs{\eta}, Z_{\bs{k}}, P_{\bs{k}} \right) \xrightarrow{\textrm{Section \ref{sec:bary-tensor}}} p(\bs{\xi}).
\end{align}
Thus, our main result below gives conditions on $p$ so that the output on the right of \eqref{eq:map-to-eta} equals the correct evaluation $p(\bs{\xi})$.
To proceed, we need a more involved deconstruction of the element identifier multi-indices $\bs{a}$ and $\bs{b}$. The particular rules in Section \ref{ssec:collapsed} that define the possible values of $\bs{a}$ and $\bs{b}$ ensure that a collection of tree structures can be constructed from $\bs{a}$ and $\bs{b}$. Let each dimension $1, 2, \ldots, d$, correspond to a node. The directed edges correspond to drawing an arrow from node $b_j$ ending at node $a_j$ for each $j = 1, \ldots, c$. Since the indices $\{a_j\}_{j=1}^c$ are all distinct, one arrow at most lands at each node, and therefore this construction forms a collection of trees. With this structure, we now define an `ancestor function' on the set of dimensions, which identifies which indices are ancestors of any dimension,
\begin{align}\label{eq:g-def}
  g_{\bs{a},\bs{b}}&: [d] \rightarrow 2^{[d]}, & g_{\bs{a},\bs{b}}(q) &\coloneqq \{q\} \bigcup \left\{ i \in [d] \;\big|\; i \textrm{ is an ancestor of } q \right\},
\end{align}
where $2^{[d]}$ denotes the power set (set of subsets) of $[d]$. Note that we have also included the index $q$ in $g_{\bs{a},\bs{b}}(q)$, so that $g_{\bs{a},\bs{b}}(q)$ is always non-empty.
The identification of $\bs{a}$ and $\bs{b}$ for typical geometries in $d = 2, 3$ is given in Table \ref{tab:elements}.
Finally, through the identification of the relation $g_T$, we can articulate which functions in $\bs{\xi}$ space are exactly evaluated via the barycentric form in $\bs{\eta}$ space.
\begin{theorem}\label{thm:bary-eval}
  With $d$, $\bs{k}$, $c$, $\bs{a}$, and $\bs{b}$ all given and fixed, define the following multi-index set:
  \begin{align}\label{eq:A-def}
    A \coloneqq \left\{ \bs{\alpha} \in \N_0^d \;\big|\; \sum_{j \in g_{\bs{a},\bs{b}}(q)} \alpha_j \leq k_q \textrm{ for every } q \in [d] \right\},
  \end{align}
   which defines a polynomial space,
  \begin{align}\label{eq:P-def}
    P = P(A) \coloneqq \left\{ \bs{\xi}^{\bs{\alpha}} \;\big|\; \bs{\alpha} \in A \right\} \subset Q_{\bs{k}}.
  \end{align}
  Then, for every $p \in P$ (a polynomial in $\bs{\xi}$ space), the procedure in \eqref{eq:map-to-eta} that utilizes the barycentric evaluation algorithm of Section \ref{sec:bary-tensor} exactly evaluates $p(\xi)$.
\end{theorem}
\begin{proof}
  Let $p \in P$, and let $\bs{\alpha} \in \N_0^d$ denote the degree of $p$, i.e., the polynomial degree of $p$ in dimension $q$ is $\alpha_q$. Since $p(\bs{\xi}) = p(D_{\bs{a},\bs{b}}(\bs{\eta}))$, then the result is proven if we can show that $\hat{p} \coloneqq p \circ D_{\bs{a},\bs{b}} \in Q_{\bs{k}}$, since the barycentric form in $\bs{\eta}$ space is exact on this space of polynomials. Note that each Duffy transformatiom $D_{\bs{a}, \bs{b}}$ defined through \eqref{eq:duffy} and \eqref{eq:xi-def} is a (multivariate) polynomial, so that $\hat{p} = p \circ D_{\bs{a},\bs{b}}$ is a polynomial, and we need to show that only its maximum degree in dimension $q$ is less than or equal to $k_q$ for every $q = 1, \ldots, d$. 

  Fixing $q \in [d]$, the degree of $\hat{p}$ in dimension $q$ is discernible from the ancestor function $g_{\bs{a},\bs{b}}$. Let $\deg_q(f)$ denote the dimension-$q$ degree of a polynomial $f$. Then the Duffy transformation definition \eqref{eq:duffy} implies that for any two distinct dimensions $i$, $j$,
  \begin{align}
    \deg_j \left(D_{i,j} \circ f \right) &= \deg_i(f) + \deg_j(f), &
    \deg_q \left(D_{i,j} \circ f \right) &= \deg_q(f), & q &\in [d] \backslash \{i\}.
  \end{align}
  Since $D_{\bs{a},\bs{b}}$ is a composition of univariate Duffy maps, the degree of $\hat{p}$ along any dimension $q$ can be determined by tracing the history of which dimensions collapse onto $q$, i.e., is determined by $g_{\bs{a},\bs{b}}(q)$. Thus, 
  \begin{align}
    \deg_q \left( \hat{p} \right) = \deg_q(p) + \sum_{i \textrm{ is an ancestor of } q} \deg_i(p) = \sum_{i \in g_{\bs{a},\bs{b}}(q)} \deg_i(p) = \sum_{i \in g_{\bs{a},\bs{b}}(q)} \alpha_i.
  \end{align}
  By assumption on the index set $A$ to which $\bs{\alpha}$ belongs, this last term is bounded by $k_q$.
%
\end{proof}
Given a tensorial grid in $Z_{\bs{k}}$ in $\bs{\eta}$ space, Theorem \ref{thm:bary-eval} precisely describes what type of $\bs{\xi}$-polynomial space membership $p$ should have so that the procedure \eqref{eq:map-to-eta} exactly evaluates $p$.

\subsection{Specializations}\label{ssec:specializations}

This section describes certain specializations of the apparatus in the previous section. Our specializations will be the two- and three-dimensional domains shown in Table \ref{tab:elements}. The goal is to show how the exactness condition of Theorem \ref{thm:bary-eval} manifests on these domains, in particular, to articulate the polynonmial space $P$ defined in \eqref{eq:P-def} on which the barycentric evaluation procedure \eqref{eq:map-to-eta} is exact. We will describe $P$ for a given degree index $\bs{k}$, and will also present a special ``isotropic" case when the number of points is the same in every dimension, i.e., when $\bs{k} = (k, k, \ldots, k)$ for some non-negative scalar integer $k$.

\subsubsection{Quadrilaterals}
We consider $d = 2$, with $c = 0$ Duffy maps. In this case, we have $\bs{\eta} = \bs{\xi}$, and both variables take values on $[-1,1]^2$. Then, given degree $\bs{k} = (k_1, k_2)$, the set $A$ in \eqref{thm:bary-eval} corresponds to all multi-indices $\bs{j}$ satisfying $\bs{j} \leq \bs{k}$. Therefore, the polynomial space $P$ in \eqref{eq:P-def} is equal to $Q_{\bs{k}}$,
\begin{align}
\begin{split}
  P = \mathrm{span} \left\{ \xi_1^{j_1} \xi_2^{j_2} \;\big|\; \bs{j} \leq \bs{k} \right\} &= Q_{\bs{k}}, \\
  P = \mathrm{span} \left\{ \xi_1^{j_1} \xi_2^{j_2} \;\big|\; \bs{j} \leq \bs{k} \right\} &= Q_{(k,k)}, \hskip 10pt (k = k_1 = k_2).
\end{split}
\end{align}

\subsubsection{Triangles}
As with quadrilateral elements we have $d = 2$, but we now take $c = 1$, and a Duffy transformation collapse defined by $\bs{a} = 1$, $\bs{b} = 2$. Then, the $\bs{\eta} \leftrightarrow \bs{\xi}$ map is given by 
\begin{align}
  \xi_1 &= \frac{1}{2} \left(1 + \eta_1\right)\left(1 - \eta_2\right), & \xi_2 = \eta_2.
\end{align}
The constraints in the definition of $A$ are given by $\alpha_1 \leq k_1$ and $\alpha_1 + \alpha_2 \leq k_2$, so that the space $P$ is 
\begin{align}
\begin{split}
  P &= \mathrm{span} \left\{ \xi_1^{j_1} \xi_2^{j_2} \;\big|\; j_1 \leq k_1, \;\; j_1 + j_2 \leq k_2 \right\}, \\
  P &= \mathrm{span} \left\{ \xi_1^{j_1} \xi_2^{j_2} \;\big|\; j_1 + j_2 \leq k \right\} = P_{k}, \hskip 10pt (k = k_1 = k_2),
\end{split}
\end{align}
where we have used $P_k$ to denote the set of bivariate polynomials of total degree at most $k$.

\subsubsection{Hexahedrons}
We now move to three dimensions so $d = 3$, and taking $c = 0$ Duffy maps, again implying that $\bs{\xi} = \bs{\eta}$. Therefore, the space $P$ on which the barycentric procedure is exact is $Q_{\bs{k}}$:
\begin{align}
\begin{split}
  P = \mathrm{span} \left\{ \xi_1^{j_1} \xi_2^{j_2} \xi_3^{j_3} \;\big|\; \bs{j} \leq \bs{k} \right\} &= Q_{\bs{k}}, \\
  P = \mathrm{span} \left\{ \xi_1^{j_1} \xi_2^{j_2} \xi_3^{j_3} \;\big|\; j_q \leq k \;\; \forall \;\; q \in [3] \right\} &= Q_{(k,k,k)}, \hskip 10pt k = k_1 = k_2 = k_3.
\end{split}
\end{align}

\subsubsection{Prisms}
As with hexahedral elements we have $d = 3$, but we now take $c = 1$, and a Duffy transformation collapse defined by $\bs{a} = 1$, $\bs{b} = 2$. Then, the $\bs{\eta} \leftrightarrow \bs{\xi}$ map is given by 
\begin{align}
  \xi_1 &= \frac{1}{2} \left(1 + \eta_1\right)\left(1 - \eta_2\right), & \xi_2 &= \eta_2, & \xi_3 &= \eta_3.
\end{align}
The constraints in the definition of $A$ are given by $\alpha_1 \leq k_1$ and $\alpha_1 + \alpha_2 \leq k_2$, so that the space $P$ is 
\begin{align}
\begin{split}
  P &= \mathrm{span} \left\{ \bs{\xi}^{\bs{j}} \;\big|\; j_1 \leq k_1, \;\; j_1 + j_2 \leq k_2, \;\; j_3 \leq k_3 \right\}, \\
  P &= \mathrm{span} \left\{ \bs{\xi}^{\bs{j}} \;\big|\; j_1 + j_2 \leq k, \;\; j_3 \leq k \right\} , \hskip 10pt (k = k_1 = k_2 = k_3).
\end{split}
\end{align}

\subsubsection{Tetrahedrons}
Also with $d = 3$ and $c=2$, a Duffy transformation collapses defined by $\bs{a} = (1, 2)$ and $\bs{b} = (2, 3)$, the $\bs{\eta} \leftrightarrow \bs{\xi}$ map is given by
\begin{align}
  \xi_1 &= \frac{1}{2} \left(1 + \eta_1\right)\left(1 - \frac{1}{2} \left(1 + \eta_2\right)\left(1 - \eta_3\right)\right), & \xi_2 &= \frac{1}{2} \left(1 + \eta_2\right) \left(1 - \eta_3 \right), & \xi_3 &= \eta_3.
\end{align}
The constraints in the definition of $A$ are given by $\alpha_1 \leq k_1$ and $\alpha_1 + \alpha_2 \leq k_2$, and $\alpha_1 + \alpha_2 + \alpha_3 \leq k_3$ so that the space $P$ is 
\begin{align}
\begin{split}
  P &= \mathrm{span} \left\{ \bs{\xi}^{\bs{j}} \;\big|\; j_1 \leq k_1, \;\; j_1 + j_2 \leq k_2, \;\; j_1 + j_2 + j_3 \leq k_3 \right\}, \\
  P &= \mathrm{span} \left\{ \bs{\xi}^{\bs{j}} \;\big|\; j_1 + j_2 + j_3 \leq k, \right\} = P_k, \hskip 10pt (k = k_1 = k_2 = k_3).
\end{split}  
\end{align}
where we have used $P_k$ to denote the set of trivariate polynomials of total degree at most $k$.

\subsubsection{Pyramids}
Finally, we again take $d = 3$ and $c=2$, a Duffy transformation collapses defined by $\bs{a} = (1, 3)$ and $\bs{b} = (2, 3)$. Then, $\bs{\eta} \leftrightarrow \bs{\xi}$ map is given by
\begin{align}
  \xi_1 &= \frac{1}{2} \left(1 + \eta_1\right)\left(1 - \eta_3\right), & \xi_2 &= \frac{1}{2} \left(1 + \eta_2\right) \left(1 - \eta_3 \right), & \xi_3 &= \eta_3.
\end{align}
The constraints in the definition of $A$ are given by $\alpha_1 \leq k_1$ and $\alpha_1 \leq k_2$, and $\alpha_1 + \alpha_2 + \alpha_3 \leq k_3$ so that the space $P$ is 
\begin{align}
\begin{split}
  P &= \mathrm{span} \left\{ \bs{\xi}^{\bs{j}} \;\big|\; j_1 \leq k_1, \;\; j_2 \leq k_2, \;\; j_1 + j_2 + j_3 \leq k_3 \right\}, \\
  P &= \mathrm{span} \left\{ \bs{\xi}^{\bs{j}} \;\big|\; j_1 + j_2 + j_3 \leq k, \right\} = P_k, \hskip 10pt (k = k_1 = k_2 = k_3).
\end{split}
\end{align}

\subsection{Derivatives and gradients}\label{ssec:multid-deriv}
The results of Section \ref{ssec:poly-eval} lead naturally to derivative evaluations. In particular, by defining the function $\hat{p} \coloneqq p \circ D_{\bs{a},\bs{b}}$ in $\bs{\eta}$ space, and writing $p(\bs{\xi}) = \hat{p}(\bs{\eta}(\bs{\xi}))$, we can translate derivatives of $p$ to those of $\hat{p}$, which can be efficiently evaluated using the results from previous sections.

Using the chain rule, the gradient of $p$ can be written in terms of the gradient of $\hat{p}$,
\begin{align}
  \nabla_{\bs{\xi}} p\left(\bs{\xi}\right) = \frac{\mathrm{D} \bs{\eta}}{\mathrm{D} \bs{\xi}} \nabla_{\bs{\eta}} \hat{p}\left(\bs{\eta}\right),
\end{align}
where $\nabla_{\bs{\eta}}$ is the standard $d$-variate gradient operator with respect to the Euclidean variables $\bs{\eta}$, and we have defined the $d \times d$ Jacobian matrix of the $\bs{\xi} \mapsto \bs{\eta}$ map,
\begin{align}
  \left( \frac{\mathrm{D} \bs{\eta}}{\mathrm{D} \bs{\xi}} \right)_{i,j} = \left( \frac{\mathrm{D} D^{-1}_{\bs{a},\bs{b}}(\bs{\xi})}{\mathrm{D} \bs{\xi}} \right)_{i,j} = \pfpx{\eta_i}{\xi_j}.
\end{align}
Note that the individual Duffy maps $D_{a,b}$ defined in \eqref{eq:duffy} that collapse dimension $a$ onto dimension $b$ are invertible whenever $\eta_b \neq 1$, so that the Jacobian above is well defined away from these points. The formula above shows that since the gradient of $\hat{p}$ can be evaluated efficiently through the procedures in Section \ref{ssec:multid-deriv}, so, too, can the gradient of $p$. In particular, this procedure exactly evaluates gradients (away from singularities of the Duffy transformation) if $p \in P(A)$ where $P(A)$ is given in \eqref{eq:P-def}.

Similarly, components of the Hessian of $p$ can be evaluated as,
\begin{align}
  \pnfpx{p}{\xi_i \partial \xi_j}{2} = \left(\pnfpx{\bs{\eta}}{\xi_i \partial \xi_j}{2}\right)^T \nabla_{\bs{\eta}} \hat{p} + \left( \pfpx{\bs{\eta}}{\xi_i} \right)^T \bs{H}_{\bs{\eta}}(\hat{p}) \left( \pfpx{\bs{\eta}}{\xi_j} \right),
\end{align}
where $\pnfpx{\bs{\eta}}{\xi_i \xi_j}{2} \in \R^d$ and $\pfpx{\bs{\eta}}{\xi_i} \in \R^d$ are componentwise derivatives, and $\bs{H}_{\bs{\eta}}(\hat{p})$ is the $d \times d$ Hessian of $\hat{p}$. Again, since the Hessian of $\widehat{p}$ can be efficiently evaluated through the procedures in Section \ref{ssec:multid-deriv}, the Hessian of $p$ also inherits this asypmtotic efficiency.  This procedure again exactly evaluates Hessians (away from singularities of the Duffy transformation) if $p \in P(A)$. If $p \in P(A)$, higher-order derivatives of $p$ may likewise be computed exactly from those of $\hat{p}$ using Fa\`{a} di Bruno's formula with $\mathcal{O}\left(\prod_{j=1}^d k_j\right)$ complexity stemming from the multivarite barycentric procedures described earlier.

\section{Algorithmic and implementation details}
\label{sec:implementation}

In this section, we present the implementation details of the barycentric Lagrange interpolation in terms of the data structures and algorithms involved. The implementation follows the high-level algorithms described in Sections \ref{sec:bary-tensor} and \ref{sec:duffy}, but some details differ in service of computational routine optimization. The implementation of these concepts can be accessed in the open-source spectral/{$hp$} element library {\em Nektar++} \cite{Cantwell2015,MOXEY2020107110}.

\subsection{Algorithm}\label{ssec:algorithm}
The foundation of the implementation is in the kernel that performs the barycentric interpolation itself as given in eq.~\eqref{eq:2} -- that is, it takes the coordinate of a single arbitrary point and the stored physical polynomial values at each quadrature point in the expansion and returns the interpolated value at the arbitrary point. This kernel has been templated to perform the interpolation only in a specific direction based on the integer template parameter \texttt{DIR}, and also to return the derivative value and second-derivative value by the reference parameter based on the boolean template parameters \texttt{DERIV}, and \texttt{DERIV2}. Templating here is defined in the sense of \Cpp ~templates; i.e. that these expressions are evaluated at compile time to reduce branching overheads and enable compiler inlining. For example, when \texttt{DERIV} and \texttt{DERIV2} are not required, setting these template variables to {\tt false} allows for performance gains by removing the {\tt if} branch tests from the generated object code. The reasoning behind this unifying of the physical value evaluation and derivative interpolations is that we can make use of terms computed in the physical evaluation in the derivative interpolations saving repeat calculations (cf. \eqref{eq:barydiff} and \eqref{eq:barydiff2}). An example kernel for physical, first- and second-derivative values is shown in Algorithm \ref{alg:baryEval}. 

The next important method is the tensor-product function, which constructs the tensor line/square by calling the barycentric interpolation kernel on quadrature points, and is therefore dimension dependent and operates on the reference element in the appropriate form. The 1D version  performs the barycentric interpolation directly on the provided point and returns the physical, first- and second- derivative value in the $\xi_{1}$ direction. In 2D and 3D, we chose to implement only the first-derivative to reduce overall complexity of the tensor-product method. The 2D version constructs an interpolation in the $\xi_{1}$ direction to give an intermediate step of physical values and derivative values at the expansion quadrature points in the same $\xi_{1}$ direction. The quadrature derivative values in the $\xi_{1}$ direction can then be evaluated in the $\xi_{2}$ direction to produce the single derivative value in the $\xi_{1}$ direction at the point provided. Likewise, the quadrature physical values in the $\xi_{1}$ direction can then be evaluated in the $\xi_{2}$ direction with the derivative output enabled to return both the single derivative value in $\xi_{2}$ direction and the physical value at the provided point. The tensor product in 3D is similar, except we now also consider the $\xi_{3}$ direction and so our intermediary steps consist of constructing the tensor square. Structuring it in this manner allows for a minimum number of calls to the barycentric interpolation kernel. An example tensor product function in 2D is shown in Algorithm \ref{alg:baryTensorDeriv}.

As this tensor-product function operates on the reference element that in 2D is a quadrilateral, and in 3D a hexahedron, additionally, it can be extended to non-reference shape types by collapsing coordinates and performing the correct quadrature point mapping as described in Section~\ref{ssec:collapsed}. This is achieved by overriding the existing reference element interpolation function, which evaluates the expansion at a single (arbitrary) point of the domain to also give it the capabilities to evaluate the derivative in each direction as needed. This function is a wrapper around a virtual function that is defined for each shape type and therefore allows for the shape dependent coordinate collapsing. Example structures of these functions for a triangular shape type are shown in Algorithms \ref{alg:collTri} and \ref{alg:physEval}.

\begin{algorithm}
    \caption{Example kernel for the Barycentric interpolation of a single point to provide physical and first-derivative values dependent on the template parameters provided. The $\langle\ldots\rangle$ indicates template arguments, $(\ldots)$ indicates normal arguments, and $\bullet$ is a matrix-vector multiplication operation.}
    \begin{algorithmic}[1]
        \ProcedureTemplate{BaryEvaluate}{dir, deriv = false, deriv2 = false}{$\eta$, $p$}
        \State{$A = 0, B = 0, C = 0, D = 0, E = 0, F = 0$}
        \For{each quadrature point, $z_{\textsc{dir}}$}
            \State $x = z_{\textsc{dir}} - \eta$
            \If{$x = 0$}
                \State $p_{\eta_{\textsc{dir}}} = p_{z_{\textsc{dir}}}$
                \If{\textsc{deriv2}}
                    \State $\frac{dp}{d\eta_{\textsc{dir}}} = \bm{D}_{z_{\textsc{dir}}} \bullet p$  \Comment{Use the precomputed derivative matrix, $\bm{D}_{z}$}
                    \State $\frac{d^{2}p}{d\eta_{\textsc{dir}}^{2}} = \bm{D}^{2}_{z_{\textsc{dir}}} \bullet p$  \Comment{Use the precomputed 2nd derivative matrix, $\bm{D}^{2}_{z}$}
                    \State $\bm{out} \gets p_{\eta_{\textsc{dir}}}, \frac{dp}{d\eta_{\textsc{dir}}}, \frac{d^{2}p}{d\eta_{\textsc{dir}}^{2}}$
                \ElsIf{\textsc{deriv}}
                    \State $\frac{dp}{d\eta_{\textsc{dir}}} = \bm{D}_{z_{\textsc{dir}}} \bullet p$  \Comment{Use the precomputed derivative matrix, $\bm{D}_{z}$}
                    \State $\bm{out} \gets p_{\eta_{\textsc{dir}}}, \frac{dp}{d\eta_{\textsc{dir}}}$
                \Else
                    \State $\bm{out} \gets p_{\eta_{\textsc{dir}}}$
                \EndIf
            \EndIf
            \State $t_{1} = w_{z_{\textsc{dir}}} / x$ \label{step:t1}
            \State $A = A + t_{1} * p_{z_{\textsc{dir}}}$
            \State $F = F + t_{1}$
            \If{\textsc{deriv} \textbf{or} \textsc{deriv2}}
                \State $t_{2} = t_{1} / x$
                \State $B = B + t_{2} *  p_{z_{\textsc{dir}}}$\label{step:B}
                \State $C = C + t_{2}$ \label{step:C}
                \If{\textsc{deriv2}}
                    \State $t_{3} = t_{2} / x$
                    \State $D = D + t_{3} *  p_{z_{\textsc{dir}}}$\label{step:D}
                    \State $E = C + t_{3}$ \label{step:E}
                \EndIf
            \EndIf
        \EndFor
        \State $p_{\eta_{\textsc{dir}}} = A / F$ \label{step:p1}
            \If{\textsc{deriv} \textbf{or} \textsc{deriv2}}
                \State $FF = F*F$
                \State $AC = A*C$
                \State $\frac{dp}{d\eta_{\textsc{dir}}} = (B * F - AC) / FF$ \label{step:dp1}
                \If{\textsc{deriv2}}
                    \State $\frac{d^{2}p}{d\eta_{\textsc{dir}}^{2}} = (2*D)/F - (2*E*A)/FF - (2*B*C)/FF + (2*C*AC)/(FF*F)$
                     \State $\bm{out} \gets p_{\eta_{\textsc{dir}}}, \frac{dp}{d\eta_{\textsc{dir}}}, \frac{d^{2}p}{d\eta_{\textsc{dir}}^{2}}$
                \EndIf
            \State $\bm{out} \gets p_{\eta_{\textsc{dir}}}, \frac{dp}{d\eta_{\textsc{dir}}}$
        \Else
            \State $\bm{out} \gets p_{\eta_{\textsc{dir}}}$
        \EndIf
        \EndProcedureTemplate
    \end{algorithmic}
    \label{alg:baryEval}
\end{algorithm}

\begin{algorithm}
    \caption{Sum factorization using the \texttt{BaryEvaluate} function for 2D shape types to give physical and first-derivative values. The variables $phys0$ and $deriv0$ refer to arrays which are populated with the physical and derivative values respectively, taken at each line of points denoted by index $i$ in the $\xi_{1}$ direction, as described in Section \ref{ssec:algorithm}.}
    \begin{algorithmic}[1]
        \Procedure{BaryTensorDeriv}{$\eta$, $p$}
        
        \For{each quadrature point in $\eta_2$ direction, $x_{2}$}
        		\State{$phys0[i], deriv0[i] \gets\textsc{BaryEvaluate} \langle 0, true \rangle(\eta_{1}, p_{\left(i \times k_{1}\right)})$}
        \EndFor
         \State{$\frac{dp}{d\eta_{1}} \gets\textsc{BaryEvaluate} \langle 1 \rangle(\eta_{2}, deriv0[0])$}
         \State{$p_{\eta}, \frac{dp}{d\eta_{2}} \gets\textsc{BaryEvaluate} \langle 1,true \rangle(\eta_{2}, phys0[0])$}
         \State{$\bm{out} \gets p_{\eta}, \frac{dp}{d\xi_{1}}, \frac{dp}{d\xi_{2}}$ }   
        \EndProcedure
    \end{algorithmic}
    \label{alg:baryTensorDeriv}
\end{algorithm}

\begin{algorithm}
    \caption{Coordinate collapsing for a triangle.}
    \begin{algorithmic}[1]
        \Procedure{CollapseCoords}{$\xi$}
            \If{$\xi_{y} = 1$}
                \State $\eta_{1} = -1$
                \State $\eta_{2} =1$
            \Else
                \State $\eta_{1} = 2*(1+\xi_{1}) / (1 - \xi_{2}) - 1$
                \State$\eta_{2} = \xi_{2}$
            \EndIf
            \State $\bm{out} \gets \eta$
        \EndProcedure
    \end{algorithmic}
    \label{alg:collTri}
\end{algorithm}

\begin{algorithm}
    \caption{Overview of the Barycentric solution and first-derivative evaluation for a triangle.}
    \begin{algorithmic}[1] 
        \Procedure{PhysEvaluate}{$\xi$, $p$}
            \State{$\eta \gets$\Call{CollapseCoords}{$\xi$}}
            \State $p_{\eta}, \frac{dp}{d\eta_{1}}, \frac{dp}{d\eta_{2}} \gets$ \Call{BaryTensorDeriv}{$\eta$, $p$}
            \State set up geometric factor for x derivative $G_{1} = 2 / (1 - \eta_{2})$
            \State set up geometric factor for y derivative $G_{2} = G_{1} * (\eta_{1} + 1) / 2$
            \State $\frac{dp}{d\xi_{1}}$ = $\frac{dp}{d\eta_{1}} * G_{1}$
            \State $\frac{dp}{d\xi_{2}} = \frac{dp}{d\eta_{2}} + G_{2} * \frac{dp}{d\eta_{1}}$
            \State $\bm{out} \gets p_{\xi}, \frac{dp}{d\xi_{1}}, \frac{dp}{d\xi_{2}}$
        \EndProcedure
    \end{algorithmic}
    \label{alg:physEval}
\end{algorithm}

\subsection{Complexity analysis}

Given a function $p(\eta_{\texttt{DIR}})$ evaluated at $k+1$ quadrature points $Q = \{z_0, z_1, \cdots, z_{k}\}_{\texttt{DIR}}$, we perform the following steps to find the interpolated values of $p(\eta),$ and  $\frac{\partial p }{\partial \eta_{\texttt{DIR}}}$ at a given point $\eta \notin Q$:\\

\begin{enumerate}[label=(\alph*)]
    \item Calculate and store the weights $\{w_j\}, \forall j = 0, 1, \cdots k$ as per \eqref{eq:1}, which requires storage of size $k+1$. The number of flops for this operation is $(k+1)^2 +1$, and the complexity is $O(k^2)$, which is a one-time setup cost.
    \item Calculate $p(\eta)$ using step \ref{step:p1} of Algorithm \ref{alg:baryEval}, for which we use the pre-computed weights $w$ from previous step and calculate the terms $A$ and $F$. The storage requirement for each of these terms is of size $k+1$, and the number of flops required to calculate them is  $3(k+1)$ and $2(k+1)$, respectively. (If we reuse the term $t_1$ from step \ref{step:t1}, calculating $F$ needs only $k+1$ flops). Thus, the total complexity of finding $p(\eta)$ using the barycentric method is $O(k)$, which is consistent with the evaluation in \cite{Trefethen04}.
 
    \item 
Calculate $\frac{\partial}{\partial \eta_{\texttt{DIR}}}p(\eta)$ as shown in step \ref{step:dp1} of Algorithm \ref{alg:baryEval} which uses the precomputed weights $w$ and the terms $A$ and $F$ from the previous step. Additional terms $B$ and $C$  are evaluated as per steps \ref{step:B} and \ref{step:C} of Algorithm \ref{alg:baryEval}. The storage requirement for these terms is of size $k+1$ each. The calculation of term $B$ requires $4(k+1)$ flops (or $3(k+1)$ using precomputed $t_1$). Similarly, calculating term $C$ requires $2(k+1)$ flops (or $k+1$ if we consider precomputed $t_2$). Therefore, the total complexity of evaluating $\frac{\partial }{\partial \eta_{\texttt{DIR}}}p(\eta)$ is $O(k)$.

\item 
Applying a similar analysis for $\frac{d^2}{d\eta^2_{\texttt{DIR}}}p(\eta)$, we need to evaluate additional terms $D$ and $E$ as shown in steps \ref{step:D} and \ref{step:E} of Algorithm \ref{alg:baryEval}. We need additional storage of size $k+1$ for each of these terms. The computational complexity for both $D$ and $E$ is $O(k)$.
 
%
%

\end{enumerate}

Note that the analysis presented above is independent of dimensions (\texttt{DIR}). For higher dimensions, we follow the same procedure in each individual direction. For example, in 2D we require evaluation of $p$ for $\texttt{DIR}=1$ and $\texttt{DIR} = 2$. Therefore, when $\texttt{DIR} = 2$, the algorithm takes twice the amount of calculations as 1D. Thus, the complexity of the evaluation is still $O(k)$. Similarly, for the first-derivative, we need the individual evaluations $\partial p/\partial \eta_0$ and  $\partial p/\partial \eta_1$. Therefore, the computational complexity of the derivative evaluation is $O(k)$. By extension, the computational complexity for the second-derivative is also $O(k)$.
 
%
\section{Evaluation and comparison}
\label{sec:results}
\subsection{Baseline Evaluations}
To investigate how this implementation of the barycentric interpolation affects the efficiency of evaluation for physical and derivative values, a number of tests were run across various cases. The barycentric interpolation method has been implemented within the {\em Nektar++} spectral/$hp$ element framework~\cite{Cantwell2015,MOXEY2020107110} in a discontinuous Galerkin (DG) setting and is compared with two variants of the already existing standard Lagrange interpolation method, the first where the interpolation matrix is recalculated every iteration, and the second where the interpolation matrix is stored across iterations, mimicking a scenario involving history points. All test cases below were carried out on a single core of a dual-socket Intel Xeon Gold 5120 system, equipped with 256GB of RAM, with the solver pinned to a specific core in order to reduce the influence of kernel core and socket reassignment mid-process.
\subsubsection{Construction of the baseline tests}
These baseline tests are constructed for the desired elemental shape using the hierarchical modified basis of Karniadakis \& Sherwin~\cite{Karniadakis2005} of order $P$  with tensor products of $P+2$ points in each direction. We make use of Gauss-Lobatto-Legendre points in noncollapsed directions and Gauss-Radau points in collapsed directions to avoid evaluation at singularities. The physical values at these points are provided by the polynomial $p(\bs{\xi}) = \xi_{1}^{2} + \xi_{2}^{2} - \xi_{3}^{2}$, which also allows for an analytical solution at any $\bs{\xi}$ for the physical and derivative values in each direction. The physical and derivative values are sampled on the constructed shape on a collocation grid that is again constructed as GLL/GLR points, like the quadrature rule. However, we use a fixed collocation grid size while varying order $P$, so that we ensure that the collocation grid is distinct from the quadrature rule in most cases. To ensure the same number of points is being sampled for all shape dimensions, we choose to use 64 total points because of the symmetry so that in 1D, it is $64^1$, 2D it is $8^2$, and 3D it is $4^3$. This creates some special considerations when the collocation grid matches exactly with the quadrature points used within the shape, which we discuss in the relevant sections below. We average the timings, in 1D from $10^6$ evaluations, and in 2D/3D from $10^5$ evaluations, to ensure results are not affected by system noise or other external factors. The tests are performed for a range of basis orders from $2$ to $20$. In 1D, we calculate the physical, first and second derivative values, whereas in 2D and 3D, we calculate only the physical and first derivative values.
\subsubsection{1D barycentric interpolation and derivatives}
Figure~\ref{fig:seg} shows that recalculating the interpolation matrix every cycle is the notably slower of the three methods, whereas the barycentric interpolation and stored interpolation matrix method are closer in performance with the barycentric interpolation being on average across the orders $33\%$ slower for the solution evaluation only, $20\%$ slower when including first derivatives, and $18\%$ slower when including second derivatives. However, the barycentric interpolation wins in terms of the storage complexity.  This is because the former requires $\mathcal{O}(k)$ storage to store the weights $w_j$, where $k$ and $z_j$ are given and fixed.  On the other hand, the stored interpolation matrix has the best case space complexity of $\mathcal{O}(k^2)$. An interesting feature present is the minor reductions in interpolation time for the stored matrix method at order $3$, $8$, $13$, and $18$. These basis orders correspond to the number of quadrature points being a multiple of five, which we theorize is the line cache size of the CPU being used. A match of this cache size with the quadrature point array sizes in our implementation will result in memory optimizations for the interpolation matrix multiplications. This phenomena will also be present for the recalculated matrix method, however, the result of optimization in this case is not visible on the graph due to the larger time scale. It can be seen that the baseline computational cost increases moving from interpolating the physical values only (Fig.~\ref{fig:seg}a) to also including the first derivatives (Fig.~\ref{fig:seg}b), and then the second derivatives (Fig.~\ref{fig:seg}c) for all methods.
\begin{figure}[H] 
         \centering
         \includegraphics[width=\textwidth]{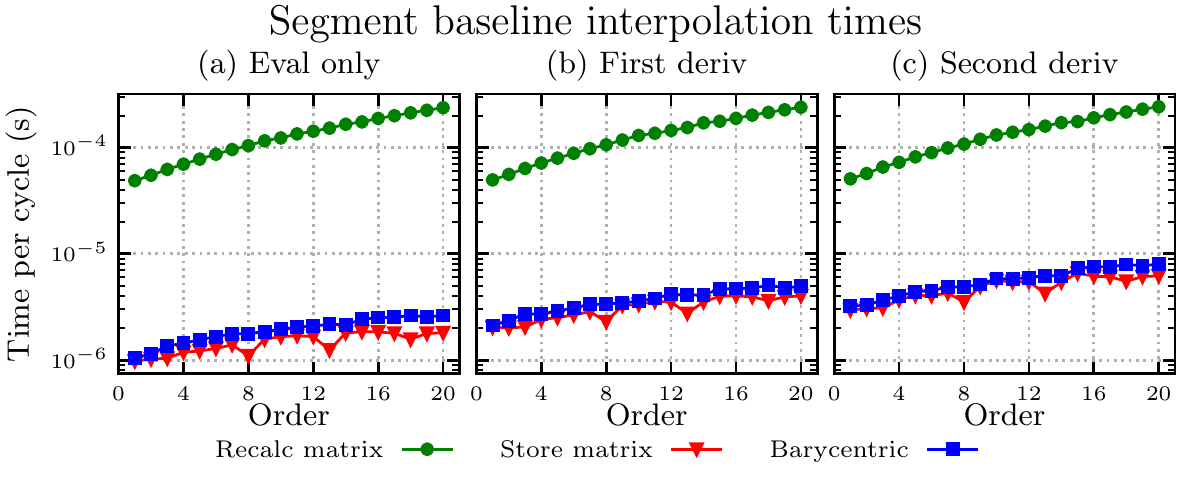}
         \caption{Baseline interpolation timings for a segment. (a) Only physical values, (b) Physical and first derivative values, (c) Physical and first-, and second-derivative values.}
         \label{fig:seg}
\end{figure}
\subsubsection{Extension to traditional tensor-product expansions}
For the traditional tensor-product expansions, we now consider a quadrilateral element in 2D and a hexahedron in 3D. Figure~\ref{fig:quad} shows the results for the quadrilateral element. Generally results are similar to that for the segment. An obvious unique feature is the spike in the recalculated matrix timing result at order $6$. The spike corresponds to $8$ quadrature points in each direction, which is the same as the number we are sampling on, and therefore the points are collocated. Consequently, the routine in which the interpolation matrix is constructed has to handle this collocation, which results in the increased cost. We can also see the the barycentric interpolation method handling this collocation, resulting in a speed-up compared to neighboring orders, evident in Figure~\ref{fig:quad}a. On average the barycentric interpolation method is $30\%$ slower across the orders for the solution evaluation only when compared to the stored interpolation matrix method. Figure~\ref{fig:quad}b including the first derivative evaluations shows that the barycentric interpolation method is on average $15\%$ faster across the orders than the stored matrix variant of the Lagrangian method indicating the computational cost savings from unifying the derivative call.
\begin{figure}[H] 
         \centering
         \includegraphics[width=\textwidth]{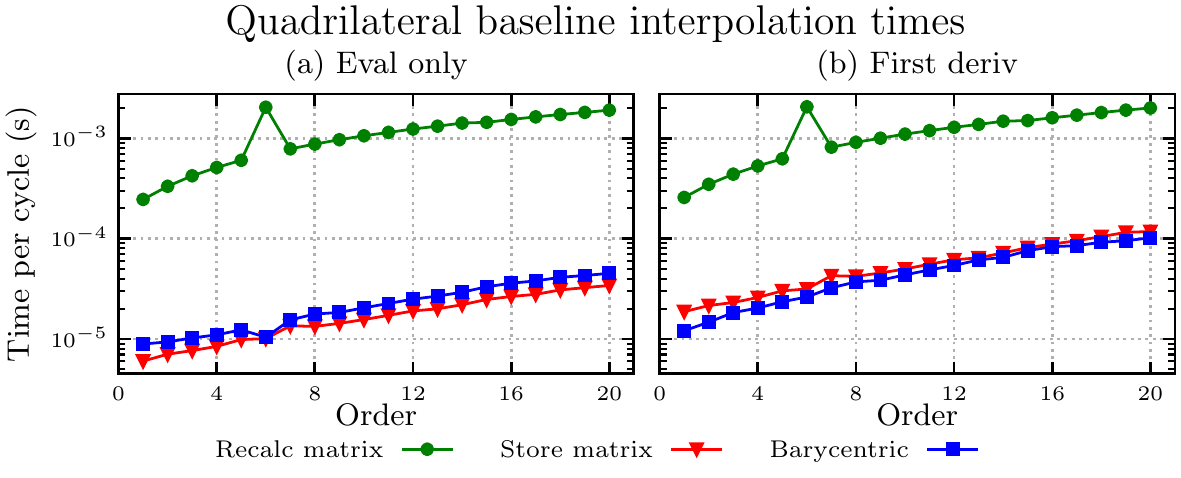}
         \caption{Baseline interpolation timings for a quadrilateral. (a) Only physical values, (b) Physical and first-derivative values.}
         \label{fig:quad}
\end{figure}
The same collocation trend is present in the hexahedral element, shown in Figure~\ref{fig:hex} with the spike now present at order $2$, which corresponds with the $4$ quadrature points in each direction. The trends are similar again to the 1D and 2D results. On average across the orders for the the solution evaluation only the barycentric interpolation method is $48\%$ slower than the stored matrix interpolation method. The most notable difference compared to the 2D results is in the first derivative timings (Fig.~\ref{fig:hex}b), which when disregarding order $2$ demonstrates the barycentric interpolation method is on average $10\%$ slower at orders $\leq11$, while at orders  $>11$ it is on average $9\%$ faster.
\begin{figure}[h]
         \centering
         \includegraphics[width=\textwidth]{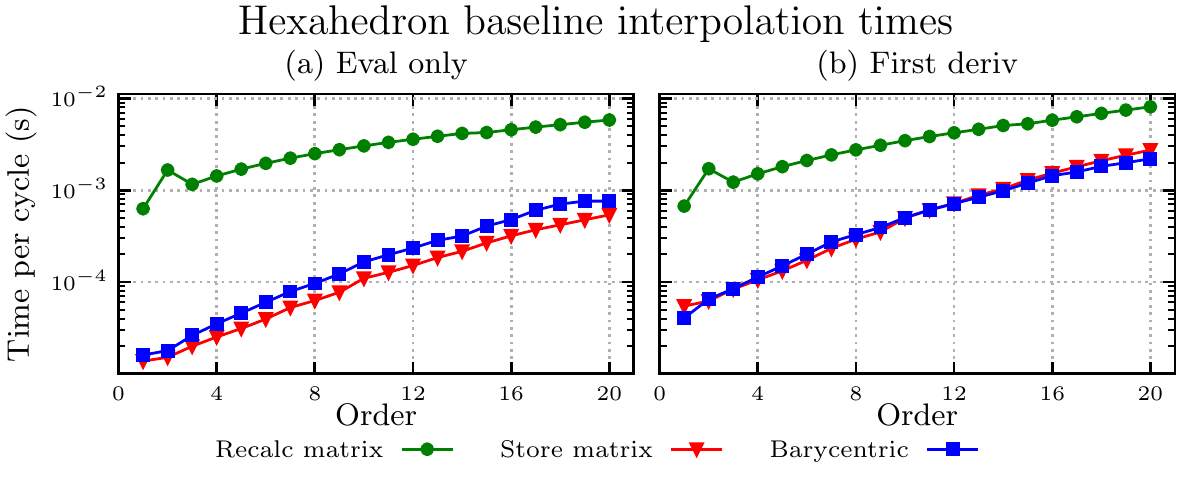}
         \caption{Baseline interpolation timings for a hexahedron. (a) Only physical values, (b) Physical and first derivative values.}
         \label{fig:hex}
\end{figure}
\subsubsection{Extension to general expansions}
We now compare the most complicated of the available shape types in two and three dimensions, the triangle and the tetrahedron, which require the use of Duffy transformations. The results shown in Figures~\ref{fig:tri} and \ref{fig:tet} align closely with their tensor-product expansion counterparts, the quadrilateral and hexahedron, respectively. A unique feature can now be seen in the recalculated matrix interpolation method that appears to show a odd/even cyclical trend. We believe this is again due to collocated points, this time as a consequence of the collapsing of the element and the routine used to calculate the interpolation matrix making use of a floor function.
\begin{figure}[H] 
         \centering
         \includegraphics[width=\textwidth]{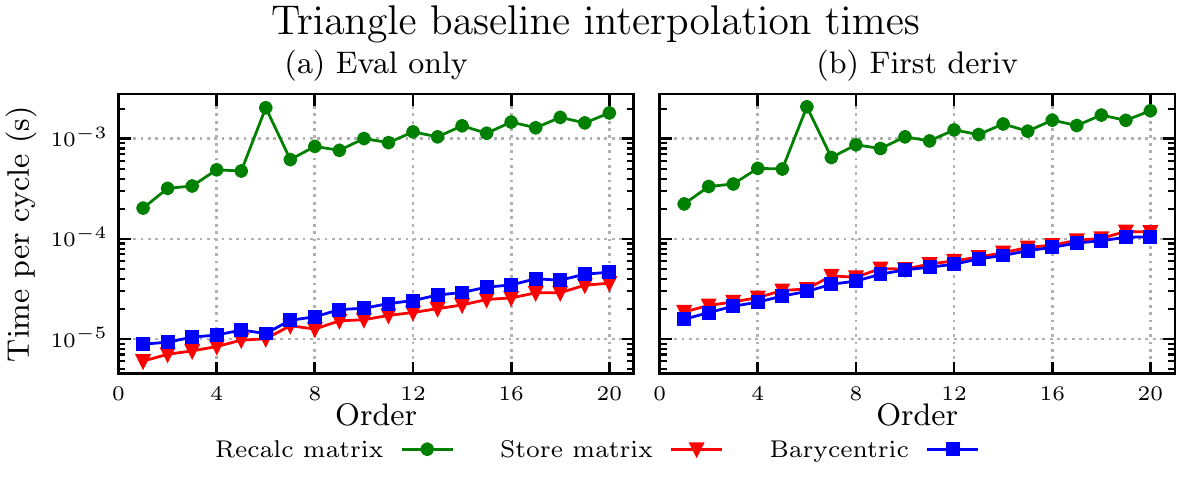}
         \caption{Baseline interpolation timings for a triangle: (a) Only physical values and (b) Physical and first-derivative values.}
         \label{fig:tri}
\end{figure}
\begin{figure}[H]
         \centering
         \includegraphics[width=\textwidth]{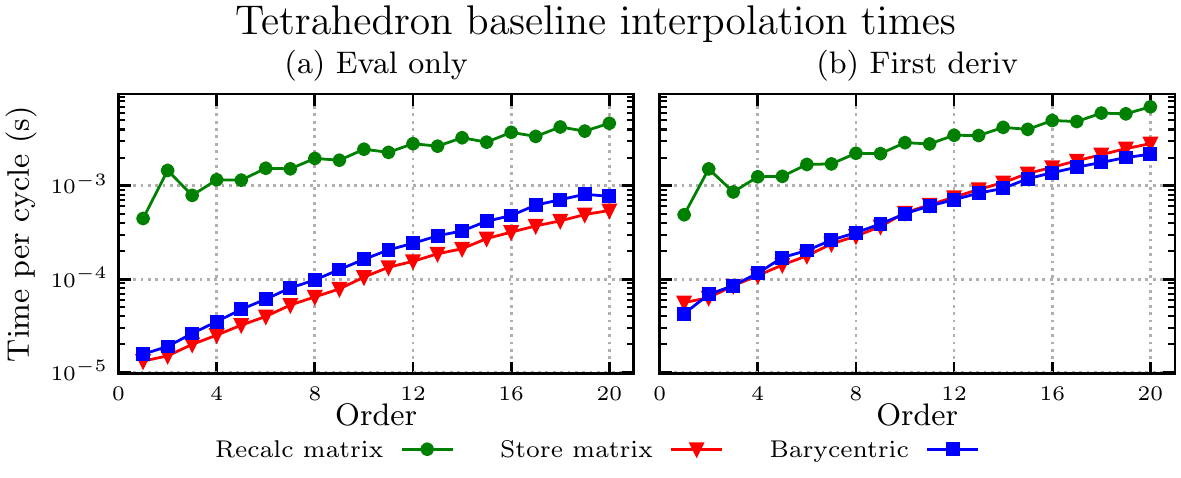}
         \caption{Baseline interpolation timings for a tetrahedron: (a) Only physical values and (b) Physical and first-derivative values.}
         \label{fig:tet}
\end{figure}
\subsubsection{Speed-up factor}
To further compare the methods, Figure~\ref{fig:factors} shows the speed-up factor when going from the recalculated matrix variant of the Lagrangian interpolation method to the barycentric interpolation method for segments, quadrilaterals, and hexahedrons. This shows that in 1D as the order increases the speed-up increases, for 2D it stays approximately the same, and for 3D it decreases. We can see that including the first-derivatives (Fig.~\ref{fig:factors}b) in 1D causes the speedup factor to reduce when compared with the evaluation only version (Fig.~\ref{fig:factors}). In 2D, the speed-up factor remains approximately consistent between the two versions, and in 3D it increases. A minimum speedup factor of approximately $7$ is observed across all tests occurring in hexahedrons greater than order $17$ when calculating the solution evaluation only.
\begin{figure}[H]
         \centering
         \includegraphics[width=\textwidth]{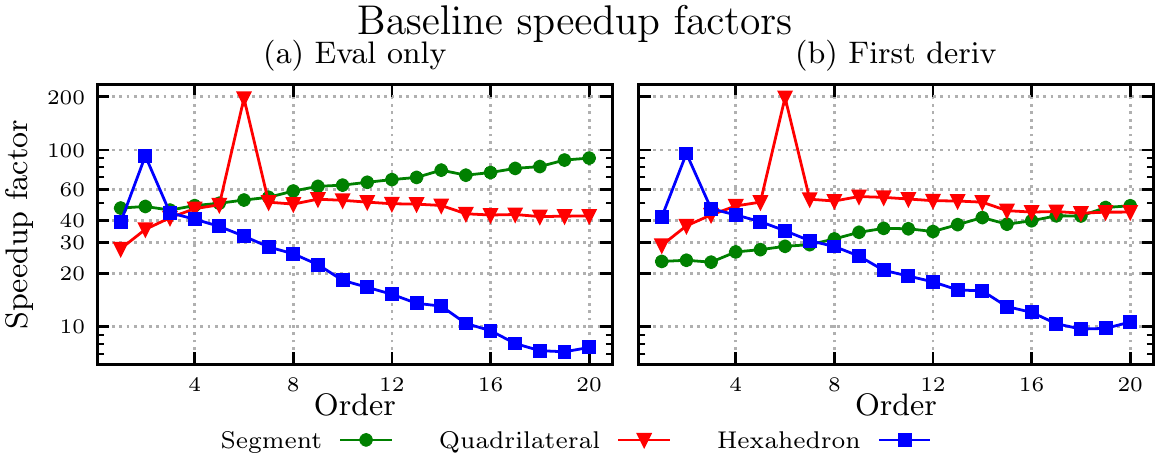}
         \caption{Speed-up factors calculated for the segment, quadrilateral, and hexahedron: (a) Only physical values and (b) Physical and first-derivative values.}
         \label{fig:factors}
\end{figure}
\subsection{Real-world usage example}
To investigate the performance of the barycentric interpolation method in a less artificial setting, we now consider a real world problem containing a nonconformal interface, again posed in a DG setting within the Nektar++ spectral/$hp$ element framework. In general we can imagine two scenarios: the first in which the nonconformal interface is fixed in time, and the second where the interface changes at each timestep to account for e.g. a grid rotation or translation. We investigate both settings in this section.
\subsubsection{Handling the nonconformal interface}
To handle the transfer of information across a nonconformal interface we adopt a point-to-point interpolation approach as outlined in \cite{Laughton21}, which involves minimizing an objective function to find an arbitrary point on a curved element edge utilizing the inverse of a parametric mapping to the reference element. In our implementation, this minimization problem is solved via a gradient-descent method utilizing a quasi-Newton search direction and backtracking line search that makes use of repeated calls to determine the physical, first- and second-derivative values within the loop. Once calculated and for a stationary interface the location of this arbitrary point in the reference element can be cached; however, to mimic a moving interface, where the minimization routine must be run every timestep, we disable this caching in order to also evaluate the performance impact of the new barycentric interpolation method. This is the equivalent of the comparison to the first Lagrange interpolation method discussed above, where the interpolation matrix is recalculated every iteration.
\subsubsection{Test case}
We select a standard linear transport equation $u_t + \nabla \cdot \bs{F}(u) = 0$ within a domain $\Omega=[-1,1]^2$, so that $\bm{F}(u) = \bm{v}u$ for a constant velocity $\bm{v} = (1,0)$, and an initial condition that is nonpolynomial, so that $\bm{u}(\bm{x},0) = \sin(2\pi x)\cos(2\pi y)$. The domain consists of a single nonconformal interface with unstructured quadrilateral subdomains on either side, as visualized in Figure~\ref{fig:domain} together with the initial condition for $u$. This means that the interpolation is being performed on the trace edges of the elements located at the nonconformal interface, which in this 2D example are segments. A polynomial order of $P=8$ is considered, and we select $Q=P+2 = 10$ quadrature points in each coordinate direction. We select a timestep size of $\Delta t = 10^{-3}$ and time for 10 cycles (i.e., $t=10$), which is the equivalent of $10^4$ timesteps.
\begin{figure}[H]
    \centering
        \includegraphics[width=0.6\textwidth]{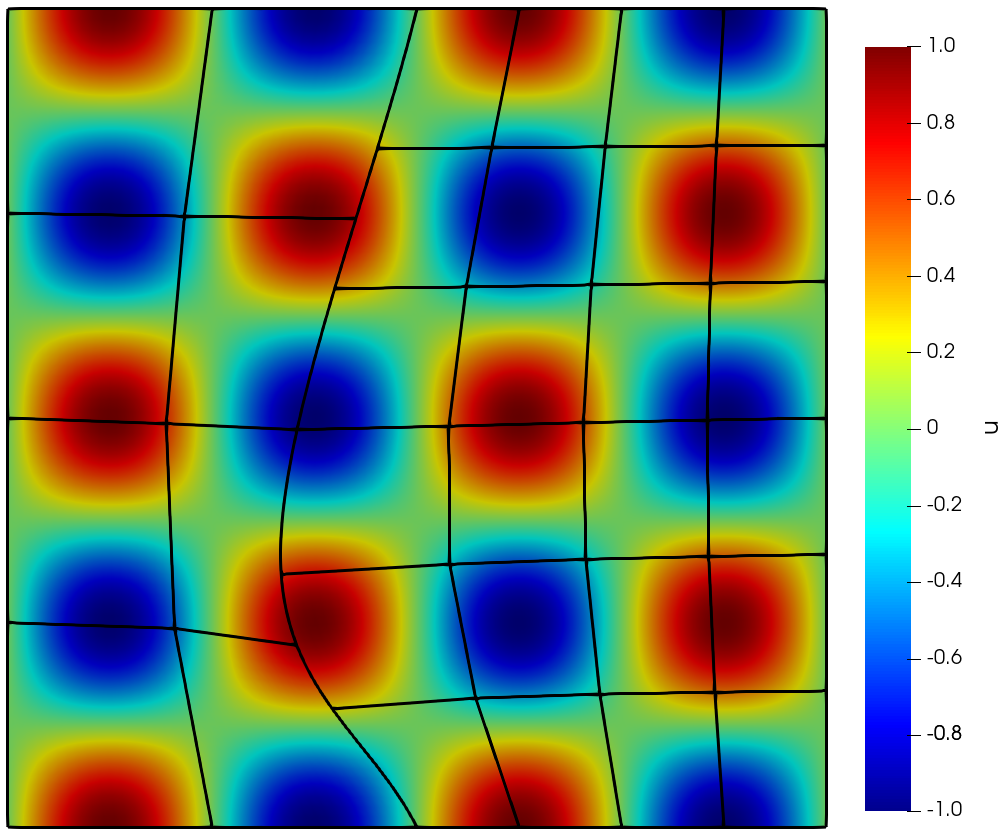}
        \caption{The nonconformal mesh for the real world usage example with the initial projection of the $u$ field overlaid.}
        \label{fig:domain}
\end{figure}
\begin{table}[H]
\sisetup{round-mode=places, round-precision=3, table-format=1.3e-1, scientific-notation = true}
\centering
\caption{Timings for the real world case using the Barycentric and Lagrangian method.}
\label{tab:timings}
\begin{tabular}{@{}lSSS@{}}
\toprule
             			& \multicolumn{3}{c}{Avg.  cost per timestep (s)}                            \\ \cmidrule(l){2-4} 
 Method     	   			 & {Cached} 		        & {Noncached}		& {Minimization}        	 \\ \midrule
Lagrangian      				& 0.00162059 			& 0.021594		&0.01997341	 \\
Barycentric            			& 0.00163341			& 0.00483588		&0.00320247	 \\
\midrule
\end{tabular}
\end{table}
We initially obtain a baseline time for both interpolation methods with the cache enabled. The cache is then disabled and both methods run again, allowing us to compare the cached (static) version with the non-cached (moving) version as a demonstration of the high computational cost incurred by calling this minimization routine every timestep. The results for the cached and noncached versions are shown in Table~\ref{tab:timings}. This demonstrates that with the cache enabled, the timings for both methods are practically identical because the minimization occurs only in the first timestep which incurs a negligible cost over this timescale. However, the non-cached results (where the minimization procedure is run every timestep) shows a slowdown of around $13\times$ for the Lagrangian method, but only $3\times$ for the barycentric method when compared with the cached results. We can then calculate the performance impact of these methods only on the minimization routine, which shows the routine using the barycentric method as around $6\times$ faster than the equivalent routine using the Lagrangian method. This is a significant speed-up, as the minimization routine accounts for a large proportion of the total computational time: for the Lagrangian method, this routine occupies 92\% of total time whereas using the barycentric approach reduces this to 66\%. The speed-up is realized in the total time taken for all $10^4$ timesteps being reduced from $216$s to $48$s.
%


%
\section{Conclusions}
\label{sec:conclusions}

In the context of spectral/$hp$ and high-order finite elements, solution expansion evaluation at arbitrary points in the domain has been a core capability needed for postprocessing operations such as visualization (streamlines/streaklines and isosurfaces) as well for interfacing methods such as mortaring. The process of evaluation of a high-order expansion at an arbitrary point in the domain consists of two parts:  determining in which particular element the point lies, and evaluating the expansion within that element. This work focuses on efficient solution expansion evaluation at arbitrary points within an element. We expand barycentric interpolation techniques developed on an interval to 2D (triangles and quadrilaterals) and 3D (tetrahedra, prisms, pyramids, and hexahedra) spectral/$hp$ element methods. We provided efficient algorithms for their implementations, and demonstrate their effectiveness using the spectral/$hp$ element library {\em Nektar++}. The barycentric method shows a minimum speedup factor of $7$ when compared with the non-cached interpolation matrix version of the Lagrangian method across all tests demonstrating a good performance uplift, culminating in an approximately $6\times$ computational time speedup for the real-world example of advection across a nonconformal interface. In the artificial tests the barycentric method exhibits slightly worse performance than the stored interpolation matrix version of the Lagrangian method when evaluating purely physical values, with slowdowns of between $10\%$ and $50\%$ across all orders dependant on element type. However if first derivatives are also required the barycentric method can outperform the stored interpolation matrix method by up to $35\%$.

\newpage
\bibliographystyle{spmpsci}
\bibliography{references}

\end{document}